\documentclass[10pt,reqno]{amsart}
\usepackage{amsmath, amssymb}
\usepackage{amsfonts}
\usepackage{mathrsfs}
\usepackage[arrow,matrix,curve,cmtip,ps]{xy}
\usepackage{graphicx}
\usepackage{comment}
\usepackage{hyperref}
\usepackage[utf8]{inputenc}
\usepackage[capitalise]{cleveref}
\usepackage{tikz}

\usepackage{amsthm}

\crefname{section}{\S}{\S\S}
\Crefname{section}{\S}{\S\S}

\tikzset{node distance=2cm, auto}

\allowdisplaybreaks

\newtheorem{theorem}{Theorem}[section]
\newtheorem{lemma}[theorem]{Lemma}
\newtheorem{proposition}[theorem]{Proposition}
\newtheorem{corollary}[theorem]{Corollary}

\newtheorem*{theorem*}{Theorem}
\theoremstyle{remark}
\newtheorem{remark}[theorem]{Remark}

\newtheorem{definition}[theorem]{Definition}
\newtheorem{example}[theorem]{Example}

\numberwithin{equation}{section}

\newcommand{\R}{\mathbb{R}}

\newcommand{\cT}{\mathcal{T}}

\newcommand{\rank}{\operatorname{rank}}

\newcommand{\tr}{\operatorname{tr}}   

\newcommand{\cD}{\mathcal{D}}
\newcommand{\cM}{\mathcal{M}}
\newcommand{\cN}{\mathcal{N}}
\newcommand{\cW}{\mathcal{W}}
\newcommand{\reals}{\mathbb{R}}
\newcommand{\maxrank}{\operatorname{maxrank}}
\newcommand{\minrank}{\operatorname{minrank}}

\newcommand{\bd}{\partial}

\newcommand{\tw}{\widetilde{w}}
\newcommand{\Perm}{\operatorname{Perm}}

\def\pgf@readxyfile{%
  \read1 to \pgf@temp%
  \let\par=\pgf@savedpar%
  \edef\pgf@temp{\pgf@temp}%
  \ifx\pgf@temp\pgfutil@empty%
  \else\ifx\pgf@temp\pgf@partext%
    \pgfplotstreamstart%
    \pgfplotstreamend%
  \else%
    \expandafter\pgf@parsexyline\pgf@temp\pgf@stop%
  \fi\fi%
  \ifeof1\else\expandafter\pgf@readxyfile\fi%
}

\begin{document}

\title[Conformal Invariants on Weighted Graphs]
{Conformally Covariant Operators and Conformal Invariants on Weighted Graphs}

\author[D. Jakobson]{Dmitry Jakobson}
\email{jakobson@math.mcgill.ca}

\author{Thomas Ng}
\email{thomas.ng2@mail.mcgill.ca}

\author{Matthew Stevenson}
\email{matthew.stevenson2@mail.mcgill.ca}

\author{Mashbat Suzuki}
\email{mashbat.suzuki@mail.mcgill.ca}

\address{Department of Mathematics and
Statistics, McGill University, 805 Sherbrooke Str. West, Montr\'eal
QC H3A 0B9, Ca\-na\-da.}

\keywords{Weighted graph, conformal structure, moduli space, conformally covariant 
operator, conformal invariant, adjacency matrix, incidence matrix, edge Laplacian, kernel, 
signature, nodal set}

\subjclass[2010]{Primary: 05C22.  Secondary: 05C50, 53A30, 53A55, 58D27, 58J50}

\thanks{D. J. was supported by NSERC and FQRNT. T. Ng was supported by ISM and CRM. 
M. St. was supported by NSERC. M. Suz. was supported by McGill Faculty of Science.}

\date{\today}

\begin{abstract}
Let $G$ be a finite connected simple graph. We define the moduli space of conformal structures on $G$. 
We propose a definition of conformally covariant operators on graphs, motivated by \cite{GJMS}.  
We provide examples of conformally covariant operators, which include the edge Laplacian 
and the adjacency matrix on graphs.  In the case where such an operator has a nontrivial kernel, 
we construct conformal invariants, providing discrete counterparts of several results in \cite{dima1,dima2} established for Riemannian 
manifolds.  In particular, we show that the nodal sets and nodal domains of null eigenvectors are conformal invariants. 
\end{abstract}

\maketitle

\section{Introduction: conformally covariant operators} 

Conformal transformations in Riemannian geometry preserve angles between tangent vectors at every point 
$x$ on a Riemannian manifold $M$.  A Riemannian metric $g_1$ is conformally equivalent to a metric 
$g_0$ if 
\begin{equation}\label{conf:change}
(g_1)_{ij}(x)=e^{\omega(x)}(g_0)_{ij}(x), 
\end{equation}
where $g_{ij}= g(\bd/\bd x_i,\bd/\bd x_j)$  defines the 
metric $g$ in local coordinates, and where $e^{\omega(x)}$ is a positive function on $M$ called a 
{\em conformal factor}.  A {\em conformal class} $[g_0]$ of a metric $g_0$ is the set of all metrics of the 
form $\{e^{\omega(x)}g_0(x):\omega(x)\in C^\infty(M)\}$.  The Uniformization theorem for compact Riemann 
surfaces says that on such a surface, in every conformal class there exists a metric of constant 
Gauss curvature; the corresponding statement in dimension $n\geq 3$ (solution of the Yamabe 
problem) stipulates that in every conformal class there exists a metric of constant {\em scalar} curvature. 

{\em Conformally covariant} differential operators include the Laplacian in dimension two, as well as the 
conformal Laplacian, Paneitz operator and other higher order operators in dimension $n\geq 3$.  
We refer the readers to the papers \cite{FG,graham,GJMS,Paneitz,wun} for detailed description of 
those operators. 

Their defining property is the transformation law 
under a conformal change of metric: there exist $a,b\in\reals$ such that if  $g_1$ and 
$g_0$ are related as in \eqref{conf:change}, then 

\begin{equation}\label{conf:covariant:manifold} 
P_{g_1}=e^{a\omega}P_{g_0} e^{b\omega}. 
\end{equation}

It follows easily that $\ker P_{g_1}=e^{-b\omega}\ker P_{g_0}.$ Based on this observation, the authors 
of the papers \cite{dima1,dima2} constructed several conformal invariants related to the nodal sets of 
eigenfunctions in $\ker P_{g}$ (that change sign).  In the current paper, the authors initiate the 
development of the theory of conformally covariant operators on graphs, giving several 
examples of such operators and providing discrete counterparts to several results in  \cite{dima1,dima2}.


\subsection{Differential operators on graphs} 

Let $G = (V,E)$ be a finite simple graph, i.e. it has a finite vertex set and no loops or multiple edges. A \emph{weighted graph} is a pair $(G,w)$ where 
$w \colon E \to \R_+$ is a weight function. 

A \emph{differential operator} on $G$ is a linear homomorphisms on either $\textrm{Hom}(V,\R)$, $\textrm{Hom}(E,\R)$, $\textrm{Hom}(V,E)$, or $\textrm{Hom}(E,V)$. These vector spaces are 
equipped with the $L^2$-norms
\begin{equation}
\langle f, g \rangle_V = \sum_{v \in V} f(v)g(v),
\hspace{3mm} \textrm{ and } \hspace{3mm}
\langle \tilde{f}, \tilde{g} \rangle_E = \sum_{e \in E} \tilde{f}(e) \tilde{g}(e).
\end{equation}
This extends to locally-finite graphs with countable vertex sets, where the function spaces are replaced by those functions with finite $L^2$-norms.

\begin{example}
The \emph{adjacency matrix} $A_w$ of the weighted graph $(G,w)$ is the $|V| \times |V|$ matrix given by
\begin{equation}\label{eqn:adjacency}
[A_w]_{ij} = \begin{cases}
w(v_i,v_j) & \colon (v_i,v_j) \in E,\\
0 & \colon \textrm{otherwise.}
\end{cases}
\end{equation}
The \emph{degree matrix} $D_w$ is the the diagonal matrix with $[D_w]_{ii} = \sum_{j=1}^n [A_w]_{ij}$, 
and the \emph{vertex Laplacian} is $\Delta_w = D_w - A_w$. The vertex Laplacian is an example of an elliptic 
Schr\"{o}dinger operator in the sense of \cite{ycdv}.
\end{example}

\section{Conformal changes of metric} 

Let $\mathcal{W}(G)$ be the space of all weight functions on the graph $G$. Inspired by the notion of 
conformal equivalence of Riemannian metrics on a Riemannian manifold, we define below the notion of 
conformal equivalence of weights as in \cite{bobenko, champion, glickstein, Luo}.  

\begin{definition}\label{def:weight}
Two weight functions $w, \tilde{w} \in \mathcal{W}(G)$ are \emph{conformally equivalent} if there exists 
a function $u \in \textrm{Hom}(V,\R)$ such that
\begin{equation}\label{conformal:transform}
\tilde{w}(v_i,v_j) = w(v_i,v_j) e^{u(v_i) + u(v_j)}. 
\end{equation}
\end{definition}
\noindent
We say that $u$ is the \emph{conformal factor} relating $w$ and $\tilde{w}$.
This equivalence relation allows us to partition the space of all weights $\mathcal{W}(G)$ on the graph $G$ into conformal equivalence classes. 
Given $w \in \mathcal{W}(G)$, we will denote its conformal class by $[w]$.

If $\sim_c$ denotes conformal equivalence, then let $\mathcal{W}(G)/{\sim_c}$ be the space of conformal classes of weights on $G$. We will refer to $\mathcal{W}(G)/{\sim_c}$ as the \emph{(conformal) moduli space of the graph $G$}. In \cref{sec:moduli}, we study the structure of the moduli space and characterize it explicitly for connected graphs.


\section{The space of conformal classes}\label{sec:moduli}

If $G=(V,E)$ is a finite simple graph, recall that $\mathcal{W}(G)$ is the space of weights on $G$. If $\sim_c$ denotes conformal equivalence, then let $\cM:=\mathcal{W}(G)/{\sim_c}$ be the space of conformal classes of weights on $G$. We will refer to $\cM$ as the \emph{(conformal) moduli space of $G$}.

We remark that if $|V| > |E|$ (i.e. $G$ is a tree), then $\cM$ is just a point. If $|V|=|E|$, there are different 
possibilities: an odd-length cycle $C_{2n+1}$ has only one conformal class, but an even-length cycle $C_{2n}$ 
has infinitely many.  If  $|E| > |V|$, then $\cM$ is generally nontrivial.

Enumerate the vertex set as $V = \{ v_1, \ldots, v_n\}$ and the edge set $E = \{ e_1, \ldots, e_m \}$. The (unsigned) edge-vertex incidence matrix $B = (B_{ij})$ of $G$
\begin{equation}\label{eqn:unsigned}
B_{ij} = \begin{cases}
1 & \colon \textrm{the edge $e_j$ is adjacent to the vertex $v_i$} \\
0 & \colon \textrm{otherwise}
\end{cases}
\end{equation}
Notice that $B$ is determined only by the topology (combinatorics) of the graph $G$; it does not depend
 on the weight. We will be mostly interested in calculating the rank of $B$. The following is a result of Grossman, Kulkarni, and Schochetman (see \cite{grossman}) to that effect.

\begin{theorem}\label{thm:bipartite-component}
\emph{(\cite[Thm 5.2]{grossman})}
For any graph $G=(V,E)$, let $\omega_0$ be the number of bipartite components.\footnote{A bipartite component is a connected component that is also bipartite. Equivalently, $\omega_0$ is the number of connected components of $G$ that do not contain an odd cycle.} Then, $\textrm{rank}(B) = |V|-\omega_0$.
\end{theorem}

Fix a reference weight $w_0 \in \mathcal{W}(G)$ and consider $w \in [w_0]$ with conformal factor $u \in \textrm{Hom}(V,\R)$.
Then we have a system of $|E|$ linear equations given by
\begin{equation}\label{eqn:system}
u(v_i) + u(v_j) = \ln w(v_i,v_j) - \ln w_0(v_i,v_j) \textrm{ for } (v_i,v_j) \in E.
\end{equation}
Consider the transpose of $B$ as an $\R$-linear operator\footnote{By an abuse of notation, we 
consider $\mathcal{W}(G)$ to be an $\R$-vector space of dimension $|E|$ by identifying a weight 
$w = (w_{e})_{e \in E}$ with the vector $(\ln w_{e})_{e \in E} \in \R^{|E|}$. In this way, $[1]$ is a 
linear subspace and the other conformal classes are the equivalences classes in the quotient $\mathcal{W}(G)/[1]$.}
$B^t \colon \textrm{Hom}(V,\R) \to [w_0]$ that sends $u \in \textrm{Hom}(V,\R)$ to the weight $w \in [w_0]$ defined by \cref{eqn:system}. This map is by definition surjective, so $\dim([w_0]) = \textrm{rank}(B^T) = \textrm{rank}(B)$.

Let $[1]$ denote the conformal class of the combinatorial weight, then we can identify 
$\cM=\mathcal{W}(G)/{\sim_c}$ with the $\R$-vector space $\mathcal{W}(G) /[1]$. The 
above considerations then imply that
\[
\dim(\cM) = \dim(\mathcal{W}(G)) -\dim( [1]) = |E|-\textrm{rank}(B).
\]
This discussion therefore yields the following conclusion.

\begin{theorem}\label{thm:homeo}
If $\omega_0$ is the number of bipartite components of $G$, then 
\begin{equation}
\cM=W(G)/{\sim_c} \simeq (\R_+)^{|E|-|V|+\omega_0}.
\end{equation}
\end{theorem}

\begin{remark}
Let $G_1,\ldots,G_k$ denote the connected components of $G$, then 
\[
\mathcal{W}(G)/{\sim_c} = \mathcal{W}(G_1)/{\sim_c} \times \ldots \times \mathcal{W}(G_k)/{\sim_c}.
\]
As a consequence, we can reduce to the case where $G$ is connected; in particular,  \cref{thm:bipartite-component} tells us that when $G$ is connected,
\begin{enumerate}
\item If $G$ is bipartite, $\dim(\cM) = |E|-|V| + 1$.
\item If $G$ is not bipartite, $\dim(\cM) = |E|-|V|$.
\end{enumerate}
\end{remark}

The following proposition specifies a manner of choosing a canonical representative of each 
conformal class.
\begin{proposition}\label{prop:representative}
In each conformal class $[w] \in \cM$, there exists a unique representative $\overline{w} \in [w]$ such 
that for any $v_i \in V$, 
\begin{equation}\label{normalization}
\prod_{e \sim v_i} \overline{w}(e) = 1,
\end{equation}
where $e \sim v_i$ denotes that the edge $e$ has the vertex $v_i$ as an endpoint.
\end{proposition}
\begin{remark}
The equation \eqref{normalization} is a very useful normalization condition, simplifying 
computations in examples of section  \ref{sec:partition}.  In the continuous setting 
(\cite{dima1,dima2}), a convenient normalization condition was choosing a metric with 
constant scalar curvature in each conformal class; such metrics exist by the Uniformization 
theorem in dimension $2$, and by solution of the Yamabe problem in higher dimensions. 
\end{remark}

\begin{proof}
The statement is equivalent to finding $\overline{w} \in [w]$ such that $\sum_{e \sim v_i} \ln \overline{w}(e) = 0$ for all $v_i \in V$.
Equivalently, $\ln \overline{w} \in \textrm{Ker}(B)$ i.e. $[w] \cap \textrm{Ker}(B) = \{ \overline{w} \}$. 
Notice that we have the orthogonal decomposition 
\begin{equation}
\ln\mathcal{W}(G) = \textrm{Ker}(B) \oplus \textrm{Im}(B^T),
\end{equation}
since $(\textrm{Ker}(B))^{\perp} = \textrm{Im}(B^T)$. Recall that we can identify $\ln[1]$ with 
$\textrm{Im}(B^T)$ inside $\ln\mathcal{W}(G)$, so when we pass to the quotient, we find that
\[
\ln\mathcal{W}(G)/\ln[1] = \textrm{Ker}(B).
\]
Therefore, the conformal classes are in bijection with the elements of $\textrm{Ker}(B)$; in particular, each $[w]$ intersects $\textrm{Ker}(B)$ at exactly one point.
\end{proof}

\begin{example}
Pick an edge in the even cycle $C_{2n}$ and assign to it the weight $a \in (0,\infty)$, then assign each adjacent edge the weight $\frac{1}{a}$. The following two edges will be assigned the weight $a$, and if we continue this process, we have define a weight $w_a$ on $C_{2n}$. By construction, $\prod_{e \sim v_i} w_a(e) = 1$ for each $v_i \in C_{2n}$. Conversely, given an arbitrary weight $w \in \mathcal{W}(C_{2n})$, we may compute its canonical representative as follows: put a cyclic orientation $\{ e_1, \ldots, e_m\}$ on $E$, then
\begin{equation}
a = \left( \prod_{i \colon \textrm{odd}} w(e_i) \right) \bigg/ \left( \prod_{i \colon \textrm{even}} w(e_i) \right).
\end{equation}
Thus, we have explicitly the isomorphism $\cM(C_{2n}) \stackrel{\sim}{\longrightarrow} (0,\infty)$. 
\end{example}

\begin{remark} 
It is well-known that on compact Riemann surfaces, in each conformal class there exists a unique metric of 
constant Gauss curvature.  The space of such metrics is called the {\em moduli space} of surfaces.  Riemannian   
geometry of moduli spaces has been extensively studied; in particular, the {\em Weil-Petersson} metric 
on the moduli space.  In \cite{PS}, the authors propose and study analogous notions for weighted graphs.
It seems quite interesting to relate the results in our paper to those in \cite{PS}.  The authors intend 
to consider this in a future paper. 
\end{remark} 


\section{Conformally covariant operators}\label{sec:CCO:definition}

Motivated by the transformation law \eqref{conf:covariant:manifold} for conformally covariant operators on manifolds, we define discrete conformally covariant operators below, 
and provide several examples of such operators.  The ``continuous'' transformation law 
\eqref{conf:covariant:manifold} involves pre- and post-multiplication by positive functions 
(powers of the conformal weight); on a graph, the multiplication operator by a positive function 
$f:V\to \reals$ corresponds to multiplication by the diagonal matrix 
${\mathrm diag}(f(v_1),\ldots ,f(v_n))$. 

\begin{definition}\label{def:operator}
Fix a graph $G$  and let $\{ S_w \}$ be a collection of differential operators, indexed by $w \in \mathcal{W}(G)$.  
We say that $S_w$ is {\em conformally covariant} if for any weight $\tilde{w} \in [w]$, 
there exist two invertible diagonal matrices $D_{\alpha}, D_{\beta}$ with positive entries (those 
entries should only depend on the conformal factors in \eqref{conformal:transform}), 
such that
\begin{equation}\label{transform:graphs}
S_{\tilde{w}} = D_{\alpha} S_w D_{\beta}. 
\end{equation}
In this case, we say that $S_{\tilde{w}}$ is a \emph{conformal transformation} of $S_w$.\footnote{The manner in which 
these differential operators transform under a conformal change of weight is analogous to how the GMJS operators 
transform under a conformal change of Riemannian metric in \cite{dima1, dima2}.} \\
Note that we do not require that $D_{\alpha}$ and $D_{\beta}$ have the same dimension.  
\end{definition}

\begin{example}
Write $e = (e^+,e^-)$ to denote the head and tail of an edge, and fix $F \subset E$. In \cite{PS}, the authors employ a $|E| \times |E|$ matrix, which we denote by $A_0(F,w)$, to study a Weil-Petersson type metric on $\cM$. This matrix is given by

\begin{equation}\label{eqn:ps_matrix}
[A_0(F,w)]_{ij} := \begin{cases}
w(e_i^+,e_j^+) & e_i,e_j \in F \textrm{ and $e_j^- = e_i^+$}\\
0 & \textrm{otherwise.}
\end{cases}
\end{equation}
Let $\tilde{w} \in [w]$ with conformal factor $u \in \textrm{Hom}(V,\R)$ and let $D_u$ be the diagonal matrix given by $[D_u]_{ii} = e^{u(e_i^+)+u(e_i^-)}$, then $A_0(F,\tilde{w}) = A_0(F,w) D_u$. Therefore, $A_0(F,w)$ is a conformally covariant operator on $\textrm{Hom}(E,\R)$.
\end{example}

In \cref{sec:adjacency} and \cref{sec:incidence}, we describe other classes of operators satisfying the transformation law \eqref{transform:graphs}, which include the adjacency matrix and 
the edge Laplacian.  

\subsection{Signature is a conformal invariant}\label{sec:signature}

Let $(G,w)$ be a fixed  finite simple weighted graph throughout this section.
\begin{theorem}\label{thm:kernel}
Let $S_w$ be a conformally covariant operator on $\textrm{Hom}(E,\R)$, then up to isomorphism $\textrm{ker}(S_w)$ is conformally invariant.  A similar results holds for conformally covariant operators on $\textrm{Hom}(V,\R)$.
\end{theorem}

\begin{proof}
If $\tilde{w} \in [w]$, then there exists an invertible diagonal matrices $D_{\alpha},D_{\beta}$ such that 
$S_{\tilde{w}} = D_{\alpha} S_w D_{\beta}$. 
Let $f \in \textrm{Hom}(E,\R)$, then $S_{\tilde{w}} f = 0$ iff $S_w( D_{\beta} f) = 0$. 
It follows that $\textrm{ker}(S_{\tilde{w}}) = D_{\beta}^{-1} \cdot\textrm{ker}(S_w)$. 
\end{proof}

The first assertion of \cref{thm:kernel} implies that the dimension of $\textrm{ker}(S_w)$ is a 
numerical invariant of the conformal class $[w]$ (as the graph is finite, the rank of $S_w$ must also 
then be a conformal invariant). As the multiplicity of zero eigenvalues of a matrix is the dimension 
of its nullspace, notice that
\begin{corollary}\label{cor:zero_eigenvalues}
The multiplicity of the zero eigenvalue of $S_w$ is a conformal invariant.
\end{corollary}

\begin{theorem}\label{thm:signature_invariant}
Let $S_w$ be a conformally covariant operator on $\textrm{Hom}(E,\R)$ or $\textrm{Hom}(V,\R)$. 
Then, the number of positive and negative 
eigenvalues, and the multiplicity of the zero eigenvalues of $S_w$ are conformal invariants. 
\end{theorem}

\begin{proof}
The proof is similar to the proof of the corresponding result in \cite{dima1,dima2}.  Assume 
for contradiction that there exist two weights $w_1\in [w_0]$ such that the signatures 
of $S_{w_1}$ and $S_{w_0}$ are different.  The conformal class $[w_0]$ is a path connected 
space, hence there exists a curve $w_t,t\in[0,1]$ starting at $w_0$ and ending at $w_1$. 
The eigenvalues of $S_{w_t}$ depend continuously on $t$, and the multiplicity of $0$ 
is constant by Corollary \ref{cor:zero_eigenvalues}.  That means that the number of positive 
and negative eigenvalues of $A_{w_t}$ remains constant, which is a contradiction.  
\end{proof}

Let $M$ be a matrix, then the {\em signature} $\textrm{sig}(M)$ is the triple $(N_+(M), N_0(M),N_-(M))$, where $N_+(M)$, $N_0(M)$, and $N_-(M)$ are the number of positive, zero, and negative eigenvalues of $M$ respectively. In this setting, \cref{thm:signature_invariant} says that the signature $\textrm{sig}(S_w)$ is a conformal invariant, when $S_w$ is a conformally covariant operator.

Large classes of operators, as described in \cref{sec:CCO:definition}, satisfy the above hypothesis on the transformation matrices.

Let $S_w$ be a conformally covariant operator on $\textrm{Hom}(E,\R)$ (the analogues still hold if $S_w$ acts on $\textrm{Hom}(V,\R)$). Now, order the eigenvalues of $S_w$ as
$$
\lambda_1(S_w) \leq \lambda_2(S_w) \leq \ldots \leq \lambda_{|E|}(S_w).
$$
Our previous considerations imply the following about the sign of $\lambda_1(S_w)$:
\begin{enumerate}
\item
$\lambda_1(S_w) < 0$ iff the number of negative eigenvalues of $S_w$ is greater or equal to 1.
\item
$\lambda_1(S_w) = 0$ iff the number of zero eigenvalues is $\eta > 0$ and the number of positive eigenvalues is $|E|-\eta$.
\item
$\lambda_1(S_w) > 0$ iff the number of positive eigenvalues of $S_w$ is equal to $|E|$.
\end{enumerate}
It follows that

\begin{corollary}\label{cor:sign_first_eigenvalue}
The sign of $\lambda_1(S_w)$ is a conformal invariant.
\end{corollary}

\begin{remark}
As a consequence of \cref{prop:independent1} and \cref{prop:independent2}, the above statements are vacuous for the edge Laplacian $\Delta(E,w)$ and $\Delta(\emptyset,w)$, as the dimensions of their kernels are independent of the weight $w$.
Thus, the dimension of the cycle subspace $\textrm{ker}(\Delta(E,w))$ is also independent of $w$.
\end{remark}



\section{Adjacency matrices}\label{sec:adjacency}

Recall the adjacency matrix $A_w$ associated to the weighted graph $(G,w)$, as in \cref{eqn:adjacency}. 
For any subset $F \subset E$, the \emph{generalized adjacency matrix} $A(F,w)$ is the 
$|V| \times |V|$ matrix where the sign of the entries in $A_w$ has been changed for the edges 
$e\in F$; it is given by
\begin{equation}\label{eqn:generalized_adjacency}
A(F,w)_{ij} = \begin{cases}
-[A_w]_{ij}, & \ \  (v_i,v_j) \in F,\\
[A_w]_{ij}, & \ \ \textrm{ otherwise.}
\end{cases} 
\end{equation} 

\begin{theorem}
For each $F \subset E$, $A(F,w)$ is a conformally covariant operator.
\end{theorem}
This statement remains true even in the more general case when we allow the graph to have loops.
\begin{proof} 
Let $\tilde{w}(v_i,v_j) = e^{u(v_i) + u(v_j)} w(v_i,v_j)$ for $(v_i,v_j) \in E$ and $u \in \textrm{Hom}(V,\R)$. 
Then, $A(F,\tilde{w})= D_u A(F,w) D_u$
where $D_u = \textrm{diag}(e^{u(v_1)},\ldots, e^{u(v_n)})$. 
\end{proof}

Notice that the case $F = \emptyset$ gives that $A(F,w) = A_w$, so the adjacency matrix is also conformally covariant. 
Furthermore, it follows that the ``random walk'' matrix $M_w = (D_w)^{-1} A_w$, which consists of the probability of travelling from one vertex to another along a random walk, is also conformally covariant.

\begin{remark}
The vertex Laplacian is not in general conformally covariant.
\end{remark}

\begin{remark}
A generalized adjacency matrix $A(F,w)$ remains conformally covariant if we allow the graph $G$ to contain loops. The same holds for the matrices $A_0(F,w)$ defined in \cref{eqn:ps_matrix}.
\end{remark}


\subsection{Ranks of generalized adjacency matrices}  
Let $G$ be finite simple graph, which in this section we assume to be connected.  

\begin{definition}\label{def:minmax:rank}
The {\em maximal} rank $\maxrank(G)$ is equal to $\sup_{F,w}\rank A(F,w)$. 
The {\em minimal} rank $\minrank(G)$ is equal to $\inf_{F,w}\rank A(F,w)$. 
\end{definition}
Equivalently, this is the largest (respectively, the smallest) rank of a symmetric matrix 
with zero diagonal, whose off-diagonal entries are nonzero off the corresponding edges 
are neighbours in $G$ (we do not allow zero edge weights).  

\begin{example}
Let $S_k$ be a {\em star graph} with $k+1$ vertices; a elementary calculation shows that 
$\maxrank(S_k)=\minrank(S_k)=2$.  The same holds for $G=K_{a,b}$, the complete 
bipartite graph. 
\end{example} 
Related questions are discussed e.g. in \cite{FH}, and the nullity of combinatorial graphs was discussed e.g. in \cite{nullity:survey}.

It was shown in \cite{royle} that adjacency matrices of {\em cographs} have full rank. 
Recall that $G$ is a cograph, or complement-reducible graph, iff $G$ {\em does not} have the path $P_4$ on 4 vertices as an induced subgraph.     

The maximal size of a graph whose adjacency matrix has a given rank was studied in 
\cite{KL,HP}.  

The rank of adjacency matrices of Erd\"{o}s-Renyi random graphs was studied in \cite{CostelloVu}; 
it was shown in \cite[Prop. 3.2.2]{Fernandez} that adjacency matrices of certain families of 
random graphs have full rank with probability going to $1$.  


\subsection{Partition of the space of conformal classes}\label{sec:partition}
We next discuss a partition of the set of conformal classes defined by the signature of a (generalized) 
adjacency matrix.   

\subsubsection{The case $\maxrank(G)=|V(G)|$}  
Assume that $\maxrank(G)=|V(G)|=n$.   Then there exists a choice of $F$ such that
for all weights $w$ in an open subset of $\cW$, $0$ is not an eigenvalue of the adjacency matrix $A(F,w)$; a similar statement holds on an open subset of the conformal moduli space $\cM$.    
In this section, we restrict our attention to $F$ satisfying $\sup_w\rank A(F,w)=\maxrank(G)$.  

\begin{definition}\label{def:disc:fullrank} 
Fix a subset $F\subset E$ of the set of edges of a graph $G$ satisfying $\maxrank(G)=n$.  
A {\em discriminant hypersurface} $\cD(F)_\cW$ in the weight space $w\in\cW$ is the set of 
all weights such that the generalized adjacency matrix $A(F,w)$ has eigenvalue $0$.  Since the multiplicity of zero of $A(F,w)$ is conformally invariant, this defines a hypersurface  
$\cD(F)_\cM$ in the conformal moduli space $\cM$.  
\end{definition}

\begin{example}\label{ex:hypersurface1}
Let $C_5$ be the 5-cycle with vertices $\{ v_1,\ldots,v_5\}$. Let $G_{5,2}$ be the non-bipartite graph obtained from $C_5$ by adding the edges $(v_1,v_3)$, and $(v_1,v_4)$. Remark that $\dim \cM(G_{5,2}) = 2$. The discriminant hypersurface $D(\emptyset)_{\cM}$ associated to the standard adjacency matrix $A(\emptyset, w)$ is the subset
\[
D(\emptyset)_{\cM} = \{ (a,b) \in \cM(G_{5,2}) \colon (ab)^4 = a^3 + b^3 \}.
\]
The curve $D(\emptyset)_{\cM}$ in $\cM(G_{5,2})$ is depicted in \cref{fig:discriminant_hypersurface}.
\begin{figure}[h]
    \centering
    \includegraphics[scale=0.3]{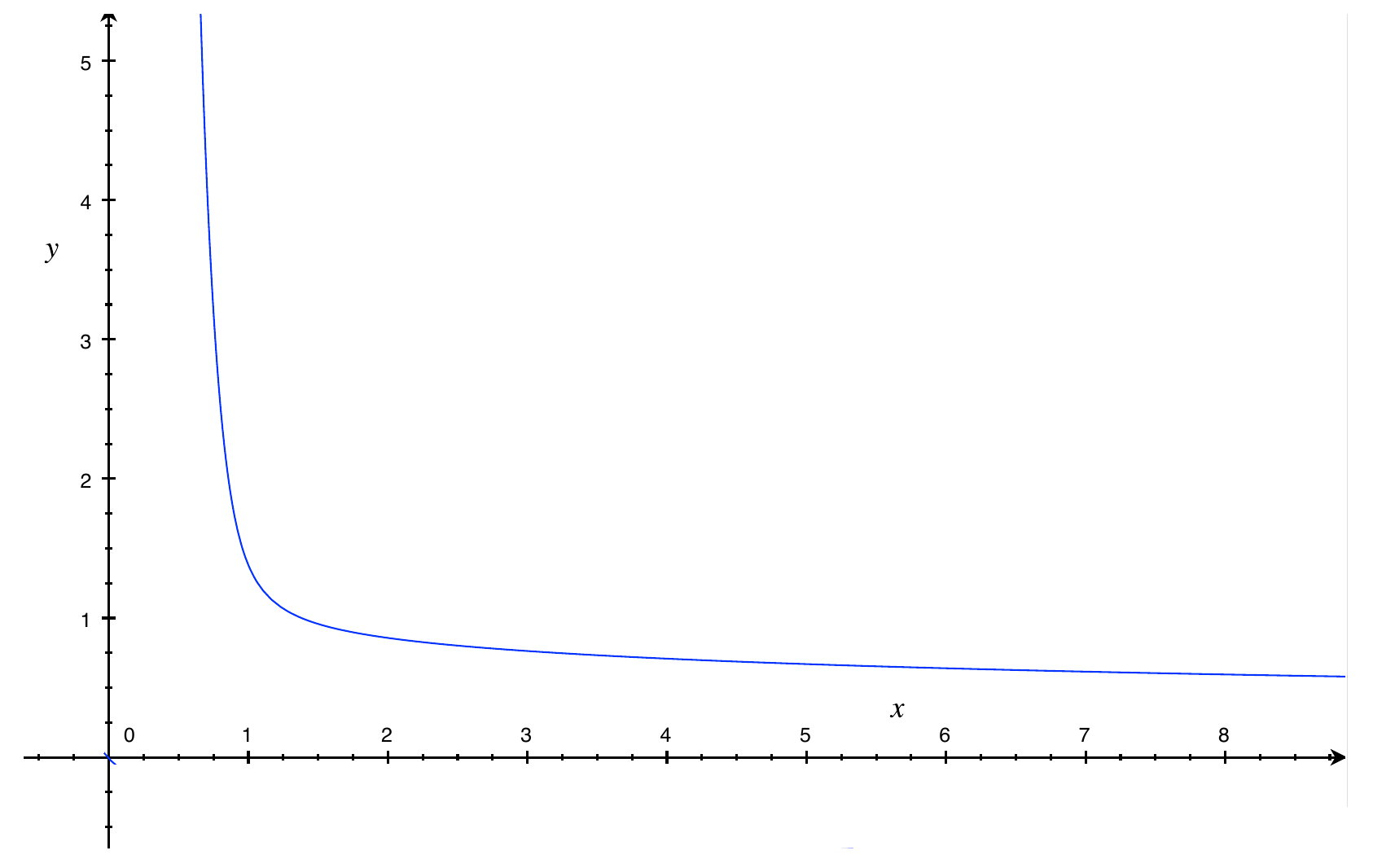}
    \caption{The discriminant hypersurface partitions $\cM(G_{5,2})$ into 2 components.}
    \label{fig:discriminant_hypersurface}
\end{figure}
\end{example}

Consider a curve $w(t)\in\cW,t\in[0,1]$.  It is clear that to change the signature of 
$A(F,w(t))$ the curve 
$w(t)$ has to cross $\cD(F)_\cW$.  We conclude that 
\begin{proposition}\label{sign:discriminant}
The signature of $A(F,w)$ does not change on connected components of 
$\cW\setminus\cD(F)_\cW$ and $\cM\setminus\cD(F)_\cM$.  
\end{proposition}

Since $\tr A(F,w)=0$ for any choice of $F\subset E$ and $w\in\cW$, we find that $A(F,w)$ 
always has at least one positive eigenvalue, and at least one negative eigenvalue.  We let  
$N_+(A)$ be the number of {\em positive} eigenvalues of $A$; and 
$N_-(A)$ be the number of {\em negative} eigenvalues of $A$.  

\begin{definition}\label{def:posneg:F}
We let $Pos_F(G)\geq 1$ to be $\sup_{w}N_+(A(F,w))$. Similarly, 
we let  $Neg_F(G)\geq 1$ to be $\sup_{w}N_-(A(F,w))$.
\end{definition}

The signature of $A(F,w)$ ranges between $(Pos_F(G),n-Pos_F(G))$ and \\ 
$(n-Neg_F(G),Neg_F(G))$.  
\begin{definition}\label{def:signature:list}
The {\em signature list} $\operatorname{List}_F(G)$ is the list of all possible signatures 
of $A(F,w)$, for different $w$.  
\end{definition}

It seems interesting to study the {\em number,}  {\em topology} and {\em geometry} 
of connected components of $\cM\setminus\cD(F)_\cM$.  

\begin{example}
Let $C_6$ be the 6-cycle with vertices $\{ v_1,\ldots,v_6\}$. Let $G_{6,3}$ be the non-bipartite graph obtained from $C_6$ by adding the edges $(v_1,v_5)$, $(v_2,v_4)$, and $(v_3,v_6)$. Remark that $\dim \cM(G_{6,3}) = 3$. Let $F = \{ (v_1,v_2)\}$, then the discriminant hypersurface $\mathcal{D}(F)_{\cM}$ associated to the generalized adjacency matrix $A(F,w)$ is the graph seen in \cref{fig:discriminant_hypersurface2}, which is cut out by the equation
\[
x^{8}y^{2}z^{2}-4({xyz})^{6}+x^{5}({2y^{2}z^{5}-2y^{5}z^{2}})+2x^{4}yz+({xyz})^{2}({y^{3}-z^{3}})^{2}+x({2yz^{4}-2y^{4}z})+1=0.
\] 
The signature of $A(F,w)$ is described by the following list:
\begin{enumerate}
\item If $w \in \cM$ is in the component whose boundary contains the origin, then the signature of $A(F, w)$ is $(3,0,3)$.
\item If $w \in \mathcal{D}(F)_{\cM}$, then the signature of $A(F,w)$ is $(3,1,2)$.
\item Otherwise, the signature of $A(F, w)$ is $(4,0,2)$.
\begin{figure}[h]
    \centering
    \includegraphics[scale=0.8]{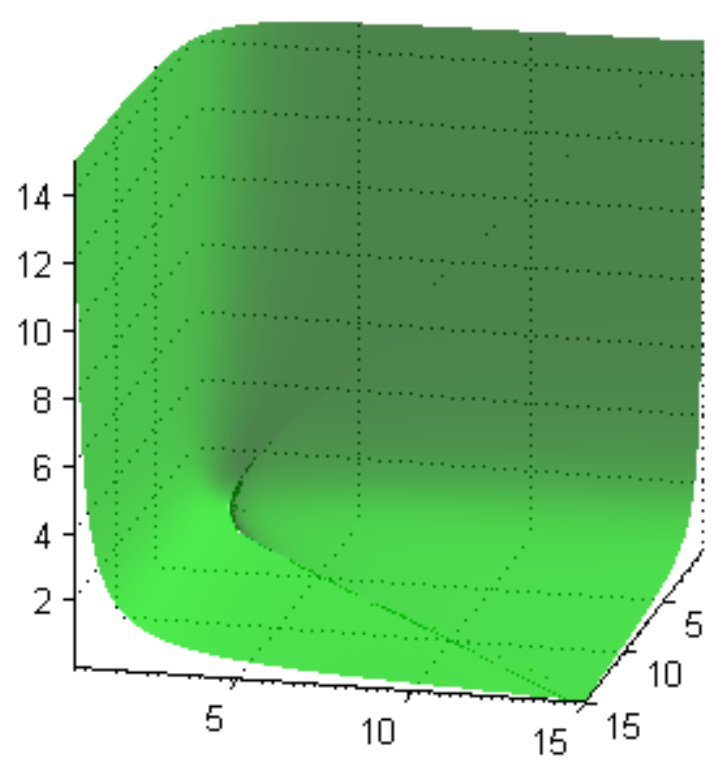}
    \caption{The discriminant hypersurface partitions $\cM(G_{6,3})$ into 2 components.}
    \label{fig:discriminant_hypersurface2}
\end{figure}
\end{enumerate}
\end{example}

We remark that the spectrum of unweighted adjacency matrix $A(G)$ has been 
studied extensively.  In particular, Graham and Pollack showed in \cite{loopswitch} that 
{\em biclique partition number} $bp(G)$ satisfies $bp(G)\geq\max\{N_+(A(G)),N_-(A(G))\}$.  

\subsubsection{The case $\maxrank(G)<V(G)$}  
In this section we consider graphs with $\maxrank(G)<V(G) = n$.  In addition, we consider 
graphs with $\maxrank(G)=n$ and subsets $F\subset E(G)$ satisfying 
$\sup_w\rank A(F,w))<\maxrank(G)$.  In those cases, $0$ is an eigenvalue of $A_F(w)$ 
for all $w\in\cW$.  

We adjust the definition of the discriminant hypersurface: 
\begin{definition}\label{def:disc:smallrank}  
A {\em discriminant hypersurface} $\cD(F)_\cW$ in the weight space $w\in\cW$ is the set of 
all weights $w$ such that the generalized adjacency matrix $A(F,w)$ satisfies 
$\rank(A(F,w))<\maxrank(G)$.  Since the multiplicity of zero of $A(F,w)$ is conformally 
invariant, this defines a hypersurface  $\cD(F)_\cM$ in the conformal moduli space $\cM$.  
\end{definition}

Proposition \ref{sign:discriminant} easily extends. Consequently, this allows one to extend the Definitions 
\ref{def:posneg:F} and  \ref{def:signature:list}. 


\subsection{Bipartite graphs} 
It is well-known that $\tr(A(F,w)^k)=\sum_{|\gamma|=k}\prod_{e\in\gamma}w(e)$; here the sum 
is taken over all closed paths of length $k$ in $G$, and the edges in $F\subset E$ have negative 
weights. In bipartite graphs, all closed paths have even length.   Accordingly, 
for any odd  $k\geq 1$ we have 
$$
\tr (A(F,w)^k)=0.  
$$
It follows that the set of eigenvalues of $A(F,w)$ is symmetric around $0$.  Accordingly, 
\begin{proposition}\label{prop:bipartite:adjacency}
Let $G$ be a bipartite simple connected graph.  Then the signature of $A(F,w)$ is always 
of the form $k,k$ for any $F\subset E$ and $w\in\cW$.  If $|V(G)|$ is {\em even} (resp. {\em odd}), then the multiplicity of $0$ as an eigenvalue of $A(F,w)$ is also {\em even} (resp. {\em odd}).  In particular, 
if $|V(G)|$ is odd, then $\ker A(F,w)$ is always nontrivial.   
\end{proposition}


\section{Incidence matrices}\label{sec:incidence}

We next describe a class of conformally covariant operators constructed using 
the (weighted) incidence matrix, or related matrices defined below.

Let $G = (V,E)$ be a finite simple graph. Let $F\subset E$ be a subset of edges which we shall {\em orient:} for each 
$e\in F$ we shall choose a {\em head vertex} $e^+\in V$ and a {\em tail vertex} $e^-\in V$.  We next define a variant 
of a well-known {\em incidence matrix} as follows.   
Enumerate the vertex set $V = \{ v_1, \ldots, v_n\}$ and the edge set $\vec{E} = \{ e_1, \ldots, e_m\}$. The signed weighted 
vertex-edge incidence matrix $M(F,w)$ (see \cite{biggs, chung} for related constructions) is an $n \times m$ 
matrix given by
\begin{equation}\label{signed:incidence} 
M(F,w)_{ij} = \begin{cases}
\sqrt{w(e_j)} & \colon \textrm{if }e_j\in F,\  v_i = e_j^+, \\
-\sqrt{w(e_j)} & \colon \textrm{if }e_j\in F,\  v_i = e_j^-, \\ 
\sqrt{w(e_j)} & \colon \textrm{if }e_j\notin F, v_i\sim e_j,\\ 
0 & \colon \textrm{if } v_i\not\sim e_j.  
\end{cases}
\end{equation}

Let $\tilde{w} \in [w]$ be a different weight in $[w]$ given by the function $u \in\textrm{Hom}(V,\R)$.  Then it is easy to show that 
\begin{equation}\label{incidence:transform} 
M(F,\tilde{w}) = M(F,w) D_u,
\end{equation}
where $D_u$ is an invertible diagonal matrix given by
\begin{equation}\label{eqn:diagonal-matrix}
(D_u)_{ii} = \begin{cases}
e^{\frac{1}{2}(u(e_i^{+})+u(e_i^{-}))} &\colon \textrm{if }e_i\in F,\ \ e_i=(e_i^{-},e_i^{+}),  \\ 
e^{\frac{1}{2}(u(v_1(i))+u(v_2(i)))} & \colon \textrm{if }e_i\not\in F,\ \  e_i=(v_1(i),v_2(i)).  
\end{cases}
\end{equation}
In other words,
\begin{theorem}
For each $F \subset E$, $M(F,w)$ is a conformally covariant operator.
\end{theorem}
The signed incidence matrix (the discrete gradient) arises from this construction when $F = E$ and the unsigned incidence matrix is $M(\emptyset, w)$. Consequently, these operators are conformally covariant.

In addition, define the following generalization of the edge Laplacian: 
\begin{equation}\label{def:signed:edge:Laplacian} 
\Delta(F,w):=M(F,w)^T\cdot M(F,w).  
\end{equation}
It follows immediately that 
\begin{theorem}\label{thm:signed:edge:Laplacian}
Let $\tilde{w} \in [w]$ be a different weight in $[w]$ given by the function $u \in \mathcal{H}_G$.  Then 
$$
\Delta(F,\tilde{w})= D_u\cdot\Delta(F,w)\cdot D_u
$$
Accordingly, $\Delta(F,w)$ is a conformally covariant operator in the sense of \eqref{transform:graphs} 
for any choice of $F\subset E$. 
\end{theorem}

We next discuss important special cases of Theorem \ref{thm:signed:edge:Laplacian} corresponding to different choices of 
$F\subset E$.  


\subsection{The edge Laplacian}\label{sec:edge-laplacian}

Consider first the case $F=E$; the operator $\Delta(E,w)$ corresponds to the weighted edge Laplacian of \cite{biggs}, so

\begin{corollary}
The edge Laplacian is a conformally covariant operator.
\end{corollary}

\begin{proposition}\label{prop:independent1}
Given a connected weighted graph $(G,w)$, $\dim \emph{ker}(\Delta(E,w)) = |E|-|V|+1$. 
In particular, $\dim \emph{ker}(\Delta(E,w))$ is independent of $w$.
\end{proposition}

\begin{proof}
Clearly, $\textrm{ker}(\Delta(E,w)) = \textrm{ker}(M(E,w))$. Now, the cycle space $\textrm{ker}(M(E,w))$ is isomorphic to the real homology $H_1(G,\R)$. For a connected graph, 
$\dim H_1(G,\R) = |E| - |E(T)|$, where $|E(T)| = |V|-1$ is the number of edges in any 
spanning tree $T$ of the combinatorial graph.
\end{proof}

\begin{remark} 
The results in this section seem to be related to the results in \cite{BG}.  The authors intend 
to study this relation further in a future paper. 
\end{remark}


\subsection{The Case $F=\emptyset$}  
The other extreme example is where $F = \emptyset$: \cref{thm:signed:edge:Laplacian} again implies that $\Delta(\emptyset,w)$ is a conformally covariant operator. 
Note that $M(\emptyset,w)$ is the unsigned weighted vertex-edge incidence matrix, and it is the weighted analogue of the matrix $B$ from \cref{sec:moduli}. 
Indeed, they have the same rank, as one is obtained from the other by scaling the columns. It follows from \cref{thm:bipartite-component} that 

\begin{proposition}\label{prop:independent2}
Given a weighted graph $(G,w)$, $\dim \emph{ker}(\Delta(\emptyset,w)) = |E|-|V|+\omega_0$, where $\omega_0$ denotes the number of bipartite components of $G$. 
In particular, $\dim \emph{ker}(\Delta(\emptyset,w))$ is independent of $w$.
\end{proposition}

\subsection{Generalized incidence matrices: left and right kernels} 
A generalized incidence matrix $M(F,w)$ will in general be a rectangular matrix; accordingly, 
we shall consider separately the {\em left kernel} of $M$: $\ker_{left} M=\{U:U\cdot M=0\}$;  
and the {\em right kernel} of $M$: $\ker_{right} M=\{V:M\cdot V=0\}$.  

In this setting, we have the `rectangular' analogue of \cref{thm:kernel}: this Proposition below follows easily from \eqref{incidence:transform}.  
\begin{proposition}\label{left:right:kernel}
Let $G$ be a simple connected graph, $F\subset E(G)$; let also $w\in\cW$ and $\tw\in[w]$.  Then 
$\ker_{left}M(F,\tw)=\ker_{left}M(F,w)$; also, $\ker_{right}M(F,\tw)=D_u^{-1}\cdot \ker_{right}M(F,w)$
\end{proposition}

\subsection{Additional variants of the edge Laplacian} 

We first prove the following elementary lemma: 
\begin{lemma}\label{omit:row}
Let the graph $G$ have $m$ edges.  
Let $J\subset\{1,2,\ldots,m\}$.  Denote by $M_J(F,w)$ the signed incidence matrix defined by 
\eqref{signed:incidence} with the {\em columns} indexed by $J$ omitted; we ignore the {\em edges}  
labelled by the elements of $J$.  Then the operator 
\begin{equation}\label{edge:laplacian_omit}
\Delta_J(F,w):=M_J(F,w)^t\cdot M_J(F,w)
\end{equation}
is also conformally covariant.  
\end{lemma}

\begin{proof}
We showed in Theorem \ref{thm:signed:edge:Laplacian} that $\Delta(F,w):=M(F,w)^t\cdot M(F,w)$ 
is conformally covariant.  The new matrix $M_J(F,w)^t\cdot M_J(F,w)$ (of order 
$(m-|J|)\cdot(m-|J|)$) is the minor $\Delta(F,w)$, with rows and columns indexed by $J$ omitted.   
Let $D_J(u)$ be the diagonal matrix that appears in Theorem \ref{thm:signed:edge:Laplacian} 
with entries corresponding to the edges labelled by $J$ omitted.  Then it follows easily from 
the definition that 
\begin{equation}\label{conf:covariant:omitJ}
\Delta_J(F,\tw)=D_J(u)\cdot \Delta_J(F,w)\cdot D_J(u), 
\end{equation}
finishing the proof. 
\end{proof}

Let $I = \{ 1,\ldots, |E|-|J|\}$, then for each pair $(i_1,i_2) \in I \times I$, let $\Lambda_{i_1,i_2}(J,F,w)$ be the matrix obtained from $\Delta_J(F,w)$ by removing the $i_1$st row and $i_2$nd column. Then,

\begin{proposition}
For each $(i_1,i_2) \in I \times I$, $\Lambda_{i_1,i_2}(J,F,w)$ is conformally covariant.
\end{proposition}

\begin{proof}
Let $D_J^{i_1}(u)$ be the matrix obtained from $D_J(u)$ by removing the $i_1$st row and $i_1$st column, and let $D_J^{i_2}(u)$ be the matrix obtained from $D_J(u)$ by removing the $i_2$nd row and $i_2$nd column. Then, it follows from \cref{omit:row} that
\[
\Lambda_{i_1,i_2}(J,F,\tilde{w}) = D_J^{i_1}(u) \cdot \Lambda_{i_1,i_2}(J,F,w) \cdot D_J^{i_2}(u).
\]
\end{proof}

The same procedure can be applied for any square subset of $I \times I$ in order to get further conformally covariant operators.


\section{Conformal invariants from Schr\"{o}dinger operators}\label{sec:inv}
Let $(G,w)$ be a finite simple weighted graph and let 
$H:V(G)\to\reals$ be a function on vertices.  
\begin{definition}\label{def:nodal:set:domain}
The {\em nodal set} $\cN(H)$ is the set of all edges $e=(v_1,v_2)\in E(G)$ such that 
$H(v_1)H(v_2)<0$, i.e. such that $H$ changes sign across $e$; together with the set of all vertices 
$v$ such that $H(v)=0$.  A {\em strong nodal domain} $U$ of $H$ is a connected subgraph of 
$G$ such that $H$ has constant sign on all the vertices of $U$.  
\end{definition}

Now, let $S$ be a conformally covariant operator (satisfying \ref{def:operator}).  Let $w\in\cW$, 
and let $\tw\in[w]$.  Then it follows from \eqref{transform:graphs} that there exists a canonical isomorphism
\begin{equation}\label{kernel:transform}
\varphi_{w, \tilde{w}} : \ker S_{\tilde{w}} \stackrel{\simeq}{\longrightarrow} \ker S_w
\end{equation}
with matrix representation  $\varphi_{w, \tilde{w}}(x) = D_\beta^{-1}x$.  

Since the entries of $D_\beta$ are all positive, the following Proposition is immediate: 
\begin{proposition}\label{nodal:set:domain:invariant}
Assume $S_w$ is conformally covariant, and $\ker S_w\neq 0$. 
Let $H\in\ker S_w$.  Then the nodal set $\cN(H)$ and strong nodal domains of $H$ 
are invariant under the isomorphism \eqref{kernel:transform}.   
\end{proposition} 

It also follows follows easily from \eqref{kernel:transform} that
\begin{proposition}\label{cor:nodal_intersection}
If $\dim \textrm{ker}(S_w) \geq 2$, then the nonempty intersection of nodal sets of 
$H_1,H_2 \in \textrm{ker}(S_w)$ and of their complements are invariant under the isomorphism \eqref{kernel:transform}. 
\end{proposition}

One can define nodal sets and strong nodal domains for functions in $\textrm{Hom}(E,\R)$, and 
prove analogues of Propositions \ref{nodal:set:domain:invariant} and \ref{cor:nodal_intersection} 
for conformally covariant operators on $\textrm{Hom}(E,\R)$.  

We can say more in the special case $S_w=A(F,w)$.    
\begin{theorem}\label{thm:edges:map} 
Let $A(F,w)$ be a generalized adjacency matrix as in \eqref{eqn:generalized_adjacency}.  
Assume that $\ker A(F,w)\neq 0$, and let $H\in\ker A(F,w)$.  Consider the map $\Psi_{H}:E(G)\to\reals$ 
defined for an edge $e=(v_1,v_2)$ by 
\begin{equation}\label{def:map:adj:kernel}
\Psi_{H}(e):=H(v_1)H(v_2)w(e)
\end{equation}
Then, $\Psi_{H}$ is invariant under the isomorphism \eqref{kernel:transform}.  
\end{theorem} 

\begin{proof}
Let $\tilde{w}\in[w]$ with conformal factor $u$.  Consider $\tilde{H} = \varphi_{w,\tilde{w}}(H)$ where \\
$\tilde{H}(v_1)=e^{-u(v_1)}H(v_1)$, and $\tilde{H}(v_2)=e^{-u(v_2)}H(v_2)$ then $\tilde{H}\in \ker A(F,\tilde{w})$. By definition $\tilde{w}(e)=e^{u(v_1)+u(v_2)}w(e)$ the desired result follows.

\end{proof}

\begin{example}
Let $G_{5,2}$ be as in \cref{ex:hypersurface1}, then along the discriminant hypersurface $\mathcal{D}(\emptyset)_{\cM}$ of the standard adjacency matrix, we have that $A(\emptyset,w)$ has a simple zero eigenvalue. 
Identifying the canonical representative of a conformal class with a pair $(a,b) \in (0,\infty)^2$, we get a `canonical' basis vector $H_{(a,b)} \in \textrm{ker} (A(\emptyset, (a,b)))$, which is given by
\[
H_{(a,b)} = \left( \frac{a^2}{b^2}, a^5-\frac{a}{b}, -a^2, -1,1 \right). 
\]
For fixed $e \in E(G_{5,2})$, we consider the range of $\psi_{H_{(a,b)}}(e)$ as we vary the conformal class along the discriminant hypersurface. Namely, consider the set
\[
X_e = \{ (a,b, \psi_{H_{(a,b)}}(e) \in (\R_+)^3 \colon (a,b) \in \mathcal{D}(\emptyset)_{\cM} \}.
\]
Projections of some $X_e$'s are pictured in \cref{fig:image_map}.
\begin{figure}[h]
    \centering
    \includegraphics[scale=0.5]{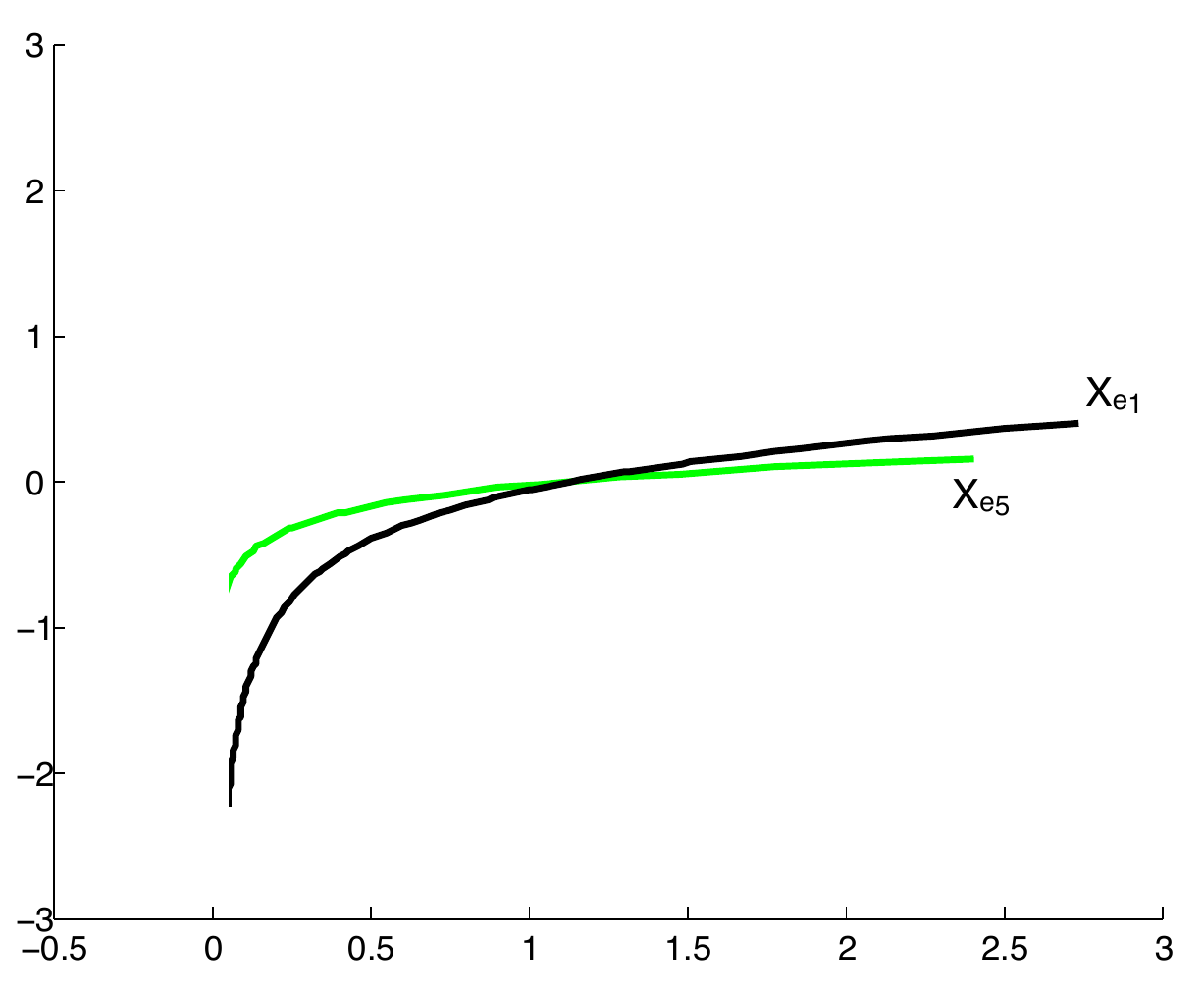}
    \caption{The horizontal axis corresponds to points $a + b$ for $(a,b) \in \mathcal{D}(\emptyset)_{\cM}$ and the curves give the value $\psi_{H_{(a,b)}}(e_1)$ and $\psi_{H_{(a,b)}}(e_5)$ respectively. In this sense, the two curves are projections of $X_{e_1}$ and $X_{e_5}$ respectively onto the plane.}
    \label{fig:image_map}
\end{figure}
\end{example}

\begin{remark}
As a consequence of \cref{prop:independent1} and \cref{prop:independent2}, the above statements are vacuous for the edge Laplacian $\Delta(E,w)$ and $\Delta(\emptyset,w)$, as the dimensions of their kernels are independent of the weight $w$.
Thus, the dimension of the cycle subspace $\textrm{ker}(\Delta(E,w))$ is also independent of $w$.
\end{remark}

Returning to the general case where $S_w$ is an arbitrary differential operator on domain $Y = V$ or $E$ with arbitrary elements $y \in Y$.
Let $f_1,\ldots,f_m$ be a basis of $\textrm{ker}(S_w)$ and $X_w = \cap_{i=1}^m f_i^{-1}(0)$.
\footnote{Note that $X_w$ is not dependent on the choice of basis.}
Define the map $\Phi_w \colon Y\backslash X_w \to \mathbb{RP}^{m-1}$
\begin{equation}
\Phi_w(y) = (f_1 : \ldots : f_m)(y) \hspace{3mm} \forall y \in Y\backslash X_w.
\end{equation}
Take $\tilde{w} \in [w]$, and let $D_{\beta}$ be the invertible diagonal matrix as in 
\ref{transform:graphs}.
Then, $D_{\beta}^{-1} f_1, \ldots, D_{\beta}^{-1} f_m$ is a basis of $\textrm{ker}(S_{\tilde{w}})$ and 
\cref{cor:nodal_intersection} implies that $X_{\tilde{w}} = X_w$; in particular, the domain of $\Phi_{\tilde{w}}$ is equal to the domain of $\Phi_{w}$. Finally, for $y \in Y\backslash X_{\tilde{w}}$,
$$
\Phi_{\tilde{w}}(y) = (D_{\beta}^{-1}(y)f_1(y) : \ldots : D_{\beta}^{-1}(y) f_m(y)) 
=(f_1(y): \ldots : f_m(y)) = \Phi_w(y),
$$
as the diagonal entries $D_{\beta}^{-1}(y)$ are nonzero. It immediately follows that

\begin{proposition}
For fixed $S_w$, the map $\Phi_w$ is a conformal invariant.
\end{proposition}

\begin{remark}
If we consider elliptic operators on Riemannian manifolds, it is well-known that the Laplacian is conformally 
covariant only in dimension two.  However, if we add a multiple of the scalar curvature as a potential, we get a {\em conformal 
Laplacian}, which is a very well-known conformally covariant operator.  Accordingly, it seems interesting 
to classify conformally covariant {\em Schr\"{o}dinger} operators (in the sense of \cite{ycdv}).
The results of this section would hold for such operators.
The authors have partial results on this classification, and plan to consider this problem further in a future paper.
In particular, if $S_w = \Delta_w + P_w$ is an elliptic Schr\"{o}dinger operator, where $\Delta_w$ is the vertex Laplacian and $P_w$ is any diagonal matrix then, under the assumption that $S_w$ is conformally covariant, the potential $P_w$ transforms as
\begin{equation}
[P_{\tilde{w}}]_{ii} + \sum_{v_j \sim v_i} \tilde{w}(v_i,v_j) = e^{2u(v_i)} \left( [P_w]_{ii} + \sum_{v_j \sim v_i} w(v_i,v_j)\right),
\end{equation}
where $\tilde{w} \in [w]$ is related to $w$ by the conformal factor $u \in \textrm{Hom}(V,\R)$.
\end{remark}


\section{{\em Determinants} and {\em Permanents} of conformally covariant operators} 

There exist polynomial graph invariants that can be constructed as determinants of certain operators 
on graphs.  For example, consider the {\em tree polynomial} 
\begin{equation}\label{tree:polynomial} 
\cT(G,w):=\sum_T\prod_{e\in T} w(e),   
\end{equation}
where the sum is taken over all spanning trees of $G$.  

Let $F=E$ (the edge set of $G$), and let $M_i(E,w)$ be the matrix defined by the formula 
\eqref{signed:incidence} with the $i$-th row (corresponding to the vertex $v_i$ of $G$) omitted.  
It can be shown (\cite[p. 135]{Kung}) that $\cT(G,w)=\det(M_i(E,w)\cdot M_i(E,w)^t)$.  Omitting 
several rows in $M(E,w)$ gives the generating polynomial for {\em rooted spanning forests 
of $G$}, see e.g. \cite[(3)]{CGS}. These results motivate the study below.

Let $A(F,w)$ be the generalized adjacency matrix of $\cref{eqn:generalized_adjacency}$.
We next define two multivariate polynomials, with the variables given by the edge weights.
\begin{definition}\label{poly:det:permanent}
Let $P_1(F,w):=\Perm A(F,w)$ and let $P_2(F,w):=\det A(F,w)$.  
\end{definition}
Recall that $\Perm X$ of a square matrix $X$ is defined by the same formula 
as $\det X$, except that the product corresponding to every permutation appears with the sign $+1$.  
By the multiplicativity of the determinant, it follows that 
$\det(D\cdot X\cdot D)=\det(D)^2\cdot \det(X)$.  
It follows easily from the definition of the permanent that for any diagonal matrix $D$, we have 
$\Perm (D\cdot X\cdot D)=\det(D)^2\cdot\Perm(X)$.  

More generally, let $\chi$ be any character of $S_{|V|}$, then by replacing
$\textrm{sgn}$ with $\chi$ in the determinant formula we get the immanant polynomials of the matrix $A(F,w)$, denoted $\textrm{Imma}_{\chi}(A(F,w))$. 
Define the following sequence of polynomials in $w = (w(e))_{e \in E}$:
\begin{equation}\label{eqn:general_polynomial}
P_{\chi}(F,w) := \textrm{Imma}_{\chi}(A(F,w)) = \sum_{\sigma \in S_{|V|}} \chi(\sigma) \prod_{i=1}^{|V|} [A(F,w)]_{i,\sigma(i)}.
\end{equation}
Take $\chi_1 = 1$ to be the trivial character and $\chi_2 = \textrm{sgn}$ to be the alternating character, then $P_{\chi_1}(F,w)$ and $P_{\chi_2}(F,w)$ correspond to the polynomials $P_1(F,w)$ and $P_2(F,w)$ respectively, as defined in \cref{poly:det:permanent} respectively.

In general, the immanant is not a multiplicative function; however, as $D(u)$ is a diagonal matrix, a simple calculation reveals that 
\[
\textrm{Imma}_{\chi_j}(D(u)\cdot A(F,w)\cdot D(u)) = \det(D(u))^2 \textrm{Imma}_{\chi_j}(A(F,w)).
\]
It follows that

\begin{theorem}\label{det:permanent} 
For each character $\chi$ of $S_{|V|}$, the polynomial $P_{\chi}(F,w)$ satisfies
$$
P_{\chi}(F,\tw)=\det(D(u))^2 P_{\chi}(F,w),
$$
where $D(u)$ is the invertible diagonal matrix such that $A(F,\tilde{w}) = D(u) A(F,w) D(u)$.
\end{theorem}

To each polynomial $P_{\chi}(F,w)$, associate a vector $x(\chi,F,w) \in \mathbb{RP}^d$ for some $d < \infty$ where the components of $x(\chi,F,w)$ are the coefficients of the monomials and the coefficient of the constant. 
As $\det(D(u))^2$ is a strictly positive real number, it follows from \cref{det:permanent} that

\begin{corollary}\label{cor:projective_vector}
Fix a character $\chi$ of $S_{|V|}$ and a subset $F \subset E$. Then, the vector $x(\chi,F,w) \in \mathbb{RP}^d$ associated to $P_{\chi}(F,w)$ is a conformal invariant.
\end{corollary}

The polynomial $P_{\chi}(F,w)$ determines a subset $Z_{\chi}(F;w)$ of $(\reals^+)^m$ as follows: 
\begin{equation}\label{zero:equation}
Z_{\chi}(F;w):=\{w\in(\reals^+)^m:P_{\chi}(F;w)=0\}.  
\end{equation} 

As $D(u)$ is an invertible diagonal matrix, Theorem \ref{det:permanent} implies the following: 
\begin{corollary} For $\chi$ and $F$ as above, 
$$
Z_j(F;\tw)=Z_j(F;w).
$$
That is, the zero set $Z_j(F;w)$ is a conformal invariant.
\end{corollary}

\begin{example}
Let $G = C_n$ where $4 | n$, and enumerate $E(G) = \{ e_1,\ldots, e_n \}$. A calculation from \cite{bien} gives that 
\begin{equation}
\textrm{det}(A(E,w)) = \left(\prod_{i \colon \textrm{even}} w(e_i) - \prod_{i \colon \textrm{odd}} w(e_i) \right)^2,
\end{equation}
where $A(E,w)$ is the usual adjacency matrix. It is then clear that the polynomial $P_2(E,w)$ has a nonempty zero locus $Z_2(E;w)$, which is a proper subset of $\mathcal{W}(C_n)$.
\end{example}

\begin{example}
Assume $G$ has an even number $n$ of vertices. Take $F = E$, then the generalized adjacency matrix $A(E,w)$ is skew-symmetric. Then, the \emph{Pfaffian} of a skew-symmetric matrix $A$ is given by
\begin{equation}
\textrm{pf}(A) := \frac{1}{2^n n!} \sum_{\sigma \in S_{n}} \textrm{sgn}(\sigma) \prod_{i=1}^n a_{\sigma(2i-1),\sigma(2i)}.
\end{equation}
It is clear that $\textrm{pf}(A(E,\tilde{w})) = \det(D(u)) \textrm{pf}(A(E,w))$. Consequently, we can again associate to $\textrm{pf}(A(E,w))$ a conformally invariant projective vector as above, and the zero locus is also conformally invariant.
\end{example}

\begin{remark}
In principle, one can consider the polynomials of \cref{eqn:general_polynomial} and the zero set of \cref{zero:equation} that arise from any conformally covariant operator. However, for certain classes of examples (in particular, the variants of the edge Laplacian given in \cref{edge:laplacian_omit}), these invariants are known to be trivial.
\end{remark}


\section{Open problems} 

\subsection{Classification of conformally covariant operators} In the present paper, we proposed 
a definition of conformally covariant operators on graphs, and provided several examples of such operators.  Motivated by \cite{GJMS}, it seems interesting to classify all conformally covariant operators on graphs (in the sense of \cite{ycdv}).  On manifolds, a very important role 
in the study of conformally covariant operators is played by the ambient space construction 
of C. Fefferman; can this construction be extended to graphs?

\subsection{Conformal moduli space} It is well-known that on compact Riemann surfaces, 
in every conformal class there exists a unique metric with constant Gauss curvature (up to 
scaling and the action of the diffeomorphism group).  For surfaces of genus $\geq 2$, 
we get the moduli space of hyperbolic metrics; its quotient by the mapping class group is 
the Teichmuller space, whose geometry and topology has been studied extensively. 
If the graph has nontrivial group $\Gamma$ automorphisms, it seems natural to consider 
the quotient $\cM/\Gamma$ of the conformal moduli space $\cM$; for many graphs $G$,  
$\Gamma(G)$ is trivial.  What is a natural analogue of the Teichmuller space for graphs?  

There exist several natural metrics on moduli spaces of surfaces, including 
the Weil-Petersson metric and the Teichmuller metric. Related problems for graphs 
have been studied in \cite{PS}.  It seems interesting to consider related metrics on $\cM$.   
The boundary of $\cM$ naturally corresponds to weights on $G$ that are $0$ on one 
or more edges; it seems interesting to describe the geometry of that boundary 
with respect to different metrics.  

Finally, some natural operations on graphs that preserve degree sequence (e.g. edge 
switches) can be realized geometrically by letting the weights of several edges decrease 
from $1$ to $0$, then letting the weights of several {\em other} edges increase from 
$0$ to $1$.  It seems that this realization would allow to ``glue'' the corresponding 
spaces of weights $\cW$ for the two graphs along a common boundary; it could 
be interesting to extend this construction to conformal moduli spaces $\cM$. The authors 
hope that this will provide some intuition for related problems on manifolds of metrics.  

\subsection{Graph Jacobians} 

In the papers \cite{BDN,BF,BN} and related articles, the authors developed discrete counterpart 
of the theory of Riemann surfaces, and explored connections to tropical geometry.  Conformal maps 
play an important role in the theory of Riemann surfaces; it seems interesting to explore connections 
between the papers cited above and the present paper. 

\subsection{Discretization, and higher-dimensional complexes}  

In the paper \cite{DP}, the authors proved that spectra of discretized Laplacian on manifolds converge 
to the spectrum of the manifold Laplacian, for suitable choices of discretized operators.  In 
\cite{bobenko, champion, glickstein, Luo} and many other papers, connections between discrete and 
continuous conformal geometry were investigated.  In \cite{wilson} the author showed that 
for a triangulated Riemann surface, and a suitable choice of inner product, the combinatorial period matrix 
converges to the (conformal) Riemann period matrix.  It seems interesting to develop a theory of conformally 
covariant operators on higher-dimensional simplicial complexes, and provide discrete counterparts to the 
results in \cite{BG} and related papers.  

\subsection{Other transformation laws}  The transformation law \eqref{transform:graphs}, 
motivated by \eqref{conf:covariant:manifold}, preserves the signature of an operator, 
leads to a simple transformation law for the kernel, and preserves the nodal set of nullvectors.  However, 
it follows from Sylvester's theorem that signature is preserved under more general 
transformations.  It could be interesting to construct operators satisfying more general 
transformation laws, to study their properties, and to possibly construct continuous analogues.

\section*{Acknowledgements}  The authors want to have Y. Canzani, R. Choksi, J.-C. Nave, S. Norin, A. Oberman, 
R. Ponge,  I. Rivin and G. Tsogtgerel for useful discussions related to the subject of this paper.  The authors want to thank the 
anonymous referee for useful remarks and corrections.

\end{document}